\documentclass[11pt,twoside]{article}

\usepackage{mathrsfs,amsfonts,amsmath,amssymb}
\usepackage{latexsym}
\newtheorem{satz}{Theorem}
\newtheorem{proposition}[satz]{Proposition}
\newtheorem{theorem}[satz]{Theorem}
\newtheorem{lemma}[satz]{Lemma}

\newtheorem{corollary}[satz]{Corollary}
\newtheorem{remark}[satz]{Remark}

\def\L{\Lambda}

\def\Z{\mathbb {Z}}
\def\F{\mathbb {F}}
\def\E{\mathsf{E}}

\def\a{\alpha}

\def\d{\delta}

\def\({\big (}
\def\){\big )}

\def\G{\Gamma}

\def\le{\leqslant}
\def\ge{\geqslant}
\def\_phi{\varphi}
\def\eps{\varepsilon}

\def\Gr{{\mathbf G}}

\def\la{\lambda}
\def\D{\Delta}

\def\R{{\mathbb R}}

\def\Sid{\mathsf{Sid}}
\def\Cay{\mathsf{Cay}}

\author{Shkredov I.D.}
\title{On an application of higher energies to Sidon sets
\footnote{This work is supported by the Russian Science Foundation under grant 19--11--00001.}
}
\date{}
\begin{document}
\maketitle

\begin{center}
	Annotation.
\end{center}

{\it \small
	We show that for any finite set $A$ and an arbitrary $\eps>0$ there is $k=k(\eps)$ such that the higher energy $\E_k(A)$ is at most $|A|^{k+\eps}$ unless $A$ has a  very specific structure. 
	As an application we obtain  that any finite subset $A$ of the real numbers or the prime field either contains an additive Sidon--type subset of size $|A|^{1/2+c}$ or a multiplicative Sidon--type subset of size $|A|^{1/2+c}$. 
}
\\

\section{Introduction} 
\label{sec:intr}

Sidon sets is  a classical object of Combinatorial Number theory, which was  introduced by S. Sidon in \cite{Sidon}.
A subset $S$ of an abelian group $\Gr$ is a {\it Sidon set} iff all its non--zero differences are distinct.  
Being "more random than random"\, this interesting class of sets were extensively studied by various authors, see \cite{AE}--\cite{Green}, \cite{Linstrom}--\cite{Ruzsa} and many others papers. 
Detailed information about Sidon sets can be found in survey \cite{Bryant}, for example.

Let $\Sid(A)$ be size of the maximal (by cardinality) Sidon subset of a set $A \subseteq \Gr$. If we want to underline the group operation we write $\Sid^{+}(A)$ or $\Sid^{\times}(A)$.
In \cite{KSS} (also, see paper \cite{Semchenkov}) it was proved that for any $A\subseteq \R$ one has 
\begin{equation}\label{f:KSS}
	\Sid(A) \ge c\sqrt{|A|} \,,
\end{equation}
where $c>0$ is an absolute constant. Of course, this result is tight (just take $A$ equals a segment of integers to see that $\Sid^{+} (A) \ll \sqrt{|A|}$). 
Oleksiy Klurman and Cosmin Pohoata  (see \cite{Klurman}) asked the following sum--product-type question (on the sum--product phenomenon see, e.g., \cite{TV}): is it true that bound \eqref{f:KSS} can be improved either for $\Sid^{+}(A)$ or for $\Sid^{\times}(A)$, where $A$ is any finite subset of the real numbers? 

\bigskip 

We write $\Sid_k (A)$ for size of maximal subset of $A$ having at most $k$ representations of any non--zero element as a difference.
Thus $\Sid_1 (A) = \Sid (A)$ and bound \eqref{f:KSS} cannot be improved for the quantity  $\Sid_k (A)$ (again take $A$ equals a segment of integers). 
Our main result is

\begin{theorem}
	Let $A\subseteq \F$ be a set, where $\F = \R$ or $\F = \F_p$ (in the prime field  case suppose, in addition, that $|A|<\sqrt{p}$, say).
	Then there are some absolute constants $c>0$, $K\ge 1$ such that
\begin{equation}\label{f:main_intr1}
	\max \{ \Sid^{+}_K (A), \Sid_K^{\times}(A) \} \gg |A|^{1/2+c} \,.
\end{equation} 
	On the other hand, for any integer $k\ge 1$ there is $A \subseteq \F$ with 
\begin{equation}\label{f:main_intr2} 
	\max \{ \Sid^{+}_k (A), \Sid^{\times}_k (A) \} \ll k^{1/2} |A|^{2/3} \,.
\end{equation}
\label{t:main_intr} 
\end{theorem}

Oliver Roche-Newton and Audie Warren obtained a bound similar to \eqref{f:main_intr2} and another estimate of the same form was obtained by Green--Peluse, see \cite{Klurman}
 (also, see \cite[page 57]{E_Extr}).
Our construction is different from these counterexamples and  we give our own proof of \eqref{f:main_intr2} at the end of section \ref{sec:proof} for completeness.

Actually, Theorem \ref{t:main_intr}  is  a consequence of a more general fact  about so--called {\it higher energies} \cite{SS1} (all required definitions can be found in section \ref{sec:definitions}). 
We think that Theorem \ref{t:Ek_intr} below is interesting in its own right and it can find further applications in Additive Combinatorics. 

\begin{theorem}
	Let $A\subseteq \Gr$ be a set, $\delta, \eps \in (0,1]$ be parameters.
	Then there is $k=k(\d, \eps)$ such that either $\E^{}_k (A) \le |A|^{k+\delta}$ or there is $H\subseteq \Gr$, $|H| \gg |A|^{\delta(1-\eps)}$, $|H+H| \ll |A|^\eps |H|$ 
	and there exists $Z\subseteq \Gr$, $|Z| |H| \ll |A|^{1+\eps}$ with $|(H\dotplus Z) \cap A| \gg |A|^{1-\eps}$. 
\label{t:Ek_intr} 
\end{theorem}

In other words, we can always take large $k$ to make $\E_k (A)$ as small as possible unless the set $A$ has a very rigid structure. 
It is easy to see (or consult \cite[Theorem 22]{Shkr_H+L}) that 
Theorem \ref{t:Ek_intr} is, actually, a criterion.

Finally, at our last section we study $k$--Sidon sets (and its generalizations), that is, sets $A$ having at most $k$ representations of any non--zero element as a difference. 
This class of sets were introduced by  Erd\H{o}s in \cite{E_Extr} and it is strongly  connected with the higher energies.
We show that such sets are even more natural than usual $B_2[g]$--sets, see the definition from \cite{Bryant}. 
In particular, size of such sets can be 
estimated 
relatively easily (unlike to $B_2[g]$--sets), they have heritability properties,  they have a natural reinterpretation in terms of its Cayley graph  and so on.

We thank Oliver Roche--Newton who communicated the question of Oleksiy  Klurman and Cosmin Pohoata to the author.
Also, we thank him for very useful comments, discussions and remarks.

\section{Definitions}
\label{sec:definitions}

By $\Gr$ we denote an abelian group.
Sometimes we underline the group operation writing $+$ or $\times$ in  the considered quantities (as the energy, the representation function and so on, see below).  
Let $\F$ be the  field $\R$ or $\F=\F_p = \Z/p\Z$ for a prime $p$.

We use the same capital letter to denote  set $A\subseteq \F$ and   its characteristic function $A: \F \to \{0,1 \}$. 
Given two sets $A,B\subset \Gr$, define  
the {\it sumset} 
of $A$ and $B$ as 
$$A+B:=\{a+b ~:~ a\in{A},\,b\in{B}\}\,.$$
In a similar way we define the {\it difference sets} and {\it higher sumsets}, e.g., $2A-A$ is $A+A-A$. 
We write $\dotplus$ for a direct sum, i.e., $|A\dotplus B| = |A| |B|$. 
For an abelian group $\Gr$
the Pl\"unnecke--Ruzsa inequality (see, e.g., \cite{TV}) holds stating
\begin{equation}\label{f:Pl-R} 
|nA-mA| \le \left( \frac{|A+A|}{|A|} \right)^{n+m} \cdot |A| \,,
\end{equation} 
where $n,m$ are any positive integers. 
We  use representation function notations like  $r_{A+B} (x)$ or $r_{A-B} (x)$ and so on, which counts the number of ways $x \in \Gr$ can be expressed as a sum $a+b$ or  $a-b$ with $a\in A$, $b\in B$, respectively. 
For example, $|A| = r_{A-A} (0)$.

For any two sets $A,B \subseteq \Gr$ the additive energy of $A$ and $B$ is defined by
$$
\E (A,B) = \E^{+} (A,B) = |\{ (a_1,a_2,b_1,b_2) \in A\times A \times B \times B ~:~ a_1+b_1 = a_2+b_2 \}| \,.
$$
If $A=B$, then  we simply write $\E^{} (A)$ for $\E^{} (A,A)$.
For $k\ge 2$ put 
\begin{equation}\label{def:E_k}
\E_k (A) = \sum_x r^k_{A-A} (x) = \sum_{\a_1, \dots, \a_{k-1}} |A\cap (A+\a_1) \cap \dots  \cap (A+\a_{k-1})|^2 \,.
\end{equation}
Clearly, $|A|^k \le \E_k (A) \le |A|^{k+1}$.
Also, we write $\hat{\E}_k (A) = \sum_x r^k_{A+A} (x)$.  

The signs $\ll$ and $\gg$ are the usual Vinogradov symbols.
When the constants in the signs  depend on a parameter $M$, we write $\ll_M$ and $\gg_M$. 
All logarithms are to base $2$.
If we have a set $A$, then we will write $a \lesssim b$ or $b \gtrsim a$ if $a = O(b \cdot \log^c |A|)$, $c>0$.
Let us denote by $[n]$ the set $\{1,2,\dots, n\}$.
By $K_{s,t}$ denote the complete subgraph with two parts of sizes $s$ and $t$.

\section{Proof of the main result}
\label{sec:proof}

We 
start with 
a
probability lemma, which was known before in the case $k=2$, see \cite{AE}, say.  
Roughly speaking, ignoring the presence of the parameter $k$ in $\Sid_{O(k)}$ from formula \eqref{f:random_Ek} below, 
Lemma \ref{l:random_Ek} works better than inequality \eqref{f:KSS}  if $\E_k (A) \le |A|^{k+\omega}$
for a certain 
$\omega<1/2$.

\begin{lemma}
	Let $A\subseteq \Gr$ be a set. 
	Then for any $k\ge 2$ one has 
	\begin{equation}\label{f:random_Ek}
		\Sid_{3k-3} (A) \gg \left( \frac{|A|^{2k}}{\E_k (A)} \right)^{1/(2k-1)} \,, 
	\quad \quad \mbox{ and } \quad \quad 
		\Sid_{2k-2} (A) \gg \left( \frac{|A|^{2k}}{\hat{\E}_k (A)} \right)^{1/(2k-1)} \,.
	\end{equation}
	\label{l:random_Ek}
\end{lemma}
\begin{proof}
	Form a random set $A_* \subseteq A$ 
	piking 
	elements of $A_*$ from $A$ independently at random   with probability $q$. 
	Then the expectation of the  solutions to the equation 
	\begin{equation}\label{f:E'_k}
		\E'_k(A) := |\{ x_1 - x'_1 = \dots = x_k -  x'_k  ~:~ x_j, x'_j \in A,\,~~ x_j, x'_j \mbox {   are different} \}| 
	\end{equation}
	is 
	$q^{2k} \E'_k (A) \le q^{2k} \E_k (A)$.
	If $q^{2k} \E_k (A) \le q|A|/2$, then we can delete roughly a half of elements of $A_*$ to 
	find a subset $B$ of $A_*$, $|B| \gg |A_*|$,  having no solutions to equation \eqref{f:E'_k}. 
	Let us take any $z\neq 0$ and prove that  $r_{B-B} (z) <  3k-2 :=g$.
	Suppose that for some $x_j,x'_j \in B$ the following holds 
\begin{equation}\label{f:r_est_k}	
	0 \neq z = x_1 - x'_1 = \dots = x_{g} -  x'_g \,. 
\end{equation}
	Clearly, any 
	$x_j$ or $x'_j$ presents in \eqref{f:r_est_k} in at most two equations and hence a pair $(x_j,x'_j)$ presents in at most three equations. 
	Also, by the definition of the set $B$ any $k$ equations from \eqref{f:r_est_k} must have some equal variables. 
	It follows that $r_{B-B} (z) \le 3\cdot (k-1)$ as required. 
	Finally, taking  $q = (|A| \E^{-1}_k (A) /2)^{1/(2k-1)}$, we obtain 
	our 
	result. 
	The second bound in \eqref{f:random_Ek} can be obtained similarly.
	In this case the correspondent  analogue of \eqref{f:r_est_k} is 
\begin{equation}\label{f:r_est_k'}
		x_1 + x'_1 = \dots = x_{g} +  x'_g 
\end{equation}
	and we see that a pair $(x_j,x'_j)$ presents in at most two equations from \eqref{f:r_est_k'}.
	Again by the definition of the set $B$ any $k$ equations from \eqref{f:r_est_k'} must have some equal variables. 
	It follows that $r_{B+B} (z) \le 2k-2$ as required.
	$\hfill\Box$
\end{proof}


\bigskip

We now obtain our driving result (Theorem \ref{t:Ek_intr} from the introduction is a particular case of Theorem \ref{t:Ek}). 
The proof is in the spirit of \cite[Theorem 21]{Shkr_H+L}.

\begin{theorem}
	Let $A\subseteq \Gr$ be a set, $\delta, \eps \in (0,1]$ be parameters, $\eps \le \delta$.\\
	$1)~$ Then there is $k=k(\d, \eps) = \exp(O(\eps^{-1} \log (1/\d)))$ such that either $\E^{}_k (A) \le |A|^{k+\delta}$ or there is $H\subseteq \Gr$, $|H| \gtrsim  |A|^{\delta(1-\eps)}$, $|H+H| \ll |A|^\eps |H|$ 
	and there exists $Z\subseteq \Gr$,  $|Z| |H| \ll |A|^{1+\eps}$ 
	with 
	$$|(H\dotplus Z) \cap A| \gg |A|^{1-\eps} \,.$$ 
	$2)~$ Similarly, either there is a set $A'\subseteq A$, $|A'| \gg |A|^{1-\eps}$ and $P\subseteq \Gr$, $|P| \gtrsim |A|^\d$ such that for all $x\in A'$ one has  $r_{A-P}(x) \gg |P| |A|^{-\eps}$ or $\E_k (A) \le |A|^{k+\d}$ with $k \ll 1/\eps$.
\label{t:Ek} 
\end{theorem}
\begin{proof} 
	Let $\E_l = \E_l (A) := |A|^{l + \kappa_l}$, where $\kappa_l \in [0,1]$ and $l\ge 2$ be any integer. 
	We assume that $\kappa_l \ge \d$.  
	By the pigeon--hole principle there is a number $\D>0$ and a set $P\subseteq \Gr$ such that $P = \{ x ~:~ \D < r_{A-A} (x) \le 2\D \}$ 
	and $\E_{l+1}  \lesssim \D^{l+1} |P|$.
	In particular, 
	\begin{equation}\label{tmp:04.03_2}
		|P| \gtrsim \E_{l+1} |A|^{-(l+1)} = |A|^{\kappa_l} \ge |A|^\d \,.
	\end{equation}
	Let $M$ be a parameter, which we will choose later in an appropriate way in each case $1)$ and $2)$.  
	Suppose that 
	\begin{equation}\label{as:l,l+1}
		\E_{l+1} \ge \frac{|A| \E_l}{M} \,.
	\end{equation}
	Then 
\begin{equation}\label{tmp:04.03_1-}
	\frac{|A| \D^l |P|}{M} \le \frac{|A| \E_l}{M} \le \E_{l+1}  \lesssim \D^{l+1} |P| \,.
\end{equation}
	Hence 
\begin{equation}\label{tmp:04.03_1}
	\frac{|A| |P|}{M} \lesssim \sum_{x\in P} r_{A-A} (x) = \sum_{a\in A} |A\cap (P+a)| \,,
\end{equation}
	and the inequality $\D \gtrsim |A|/M$ follows from \eqref{tmp:04.03_1-}. 
	Take $A' = \{ a\in A~:~ |A\cap (P+a)| \gtrsim |P|/CM \}$ with a  sufficiently large constant $C$.
	Then from \eqref{tmp:04.03_1}, we get $|A'| \gtrsim |A|/M$ and by the definition of the set $A'$, we have  $r_{A-P}(x) \gtrsim |P|/M$ for all $x\in A'$.
	To obtain $2)$ we put $M = |A|^{\eps/2}$ and we are done under assumption \eqref{as:l,l+1}.
	Now suppose that $\E_{l+1} \le |A| \E_l /M$.
	Then apply the previous  argument  to $\E_{l}$ instead of $\E_{l+1}$. 
	Again, if $\E_{l} \ge |A| \E_{l-1} /M$, then we are done otherwise $\E_{l} \le |A| \E_{l-1} /M$ and we repeat the argument. 
	Clearly, after at most $k\ll 1/\eps$ steps our algorithm finishes and we obtain $\E_k (A) \le |A|^{k+\d}$.

	It remains to obtain $1)$. 
	Returning to \eqref{tmp:04.03_1} and using the H\"older inequality several times, we get 
\[
	\frac{|A| |P|^2}{M^2} \lesssim  \E(A,P) = \sum_x r_{A-A} (x) r_{P-P} (x)\,,
\]
	and hence 
\[
	\left( \frac{|A| |P|^2}{M^2} \right)^{l+1} \lesssim \E_{l+1} \left( \sum_x r^{1+1/l}_{P-P} (x) \right)^l
	\lesssim  
	|P| \D^{l+1} \cdot \E(P) |P|^{2l-2} \,.
\]
	In other words, 
\[
	\E(P) \gtrsim \frac{|P|^3}{M^{2l+2}} := \frac{|P|^3}{M_*} \,.
\]
	By the Balog--Szemer\'edi--Gowers Theorem (e.g., see, \cite{TV}), we find $H\subseteq P$, $|H| \gg |P|/M^C_*$, $|H+H| \ll M^C_* |H|$, where $C>0$ is an absolute constant which may change from line to line.
	As above
\[
	|A| |H|/M \lesssim |H| \D \ll \sum_{x\in H} r_{A-A} (x) = \sum_{a\in A} |A\cap (H+a)| \,. 
\]
	Again, we define $W\subseteq A$ similarly to the set $A'$ and obtain in particular, that   $|W| \gtrsim |A|/M$. 
	Also, let $Z\subseteq W$ be a set such that the sets  $\{H+z\}_{z\in Z}$ form the maximal system of disjoint sets. 
	By maximality, $W\subseteq Z+H-H$ and hence by the  Pl\"unnecke inequality \eqref{f:Pl-R}, we get
\[
	|A|/M \lesssim |W| \le |Z||H-H| \ll M^C_* |Z| |H| 
\] 
	and
	thus 
\[
	|(H\dotplus Z) \cap A| = \sum_{z\in Z} |A\cap (H+z)| \gtrsim  |Z| |H|/M \gtrsim |A| / M^{C}_* \,.
\]
	In particular, 
\[
	M |A| \gtrsim |Z| |H| \,.
\]
	We now put $M= |A|^{\frac{\eps \kappa_l}{C_* l}}$, where $C_*>0$ is a sufficiently large constant. 
	Hence we get $|H+H| \ll |A|^\eps |H|$ and  $|(H\dotplus Z) \cap A| \gg |A|^{1-\eps}$.
	Also,  in view of \eqref{tmp:04.03_2} one has 
	$$|H|  \gg |P|/M^C_* \gtrsim  |A|^\delta / M^{C_*} \gtrsim  |A|^{\delta (1-\eps)} $$
	as required. 
	Now suppose that inequality \eqref{as:l,l+1} fails, i.e.,   	$\E_{l+1} \le |A| \E_l /M$.
	It gives us $\kappa_{l+1} \le \kappa_l (1-\frac{\eps}{C_* l})$ and thus iterating, we see	 that our algorithm stops after at most $\exp(O(\eps^{-1} \log (1/\d)))$ steps.
	This completes the proof. 
$\hfill\Box$
\end{proof}

\begin{remark}
	Notice that Lemma \ref{l:random_Ek}, as well as the first part of Theorem \ref{t:Ek}  work for sums as for differences and thus the first part of Theorem \ref{t:main_intr} can be obtained 
	for the quantity $\Sid^*_k$ defined via pluses but not differences.
\end{remark}


A consequence of Theorem  \ref{t:Ek} is Corollary \ref{c:pure_Sidon_intr} below, which is a direct result having no the sum--product flavour. 
Roughly speaking, it says that  basic estimate \eqref{f:KSS} can be easily improved for a wide class of sets, if we 
 ignore the presence of the parameter $K$ in formula \eqref{f:pure_Sidon_intr}.

\begin{corollary}
	Let $A\subseteq \Gr$ be a set, $\delta \in (0,1/2)$ be a parameter.
	Then either for some $K=K(\d)$, $c=c(\d)$ 
	the following holds  
	\begin{equation}\label{f:pure_Sidon_intr}
	\Sid_{K} (A) \gg |A|^{1/2+c} 
	\end{equation}
	or there is a set of shifts $T$, $|T| \gg |A|^{\d/8}$ 
	such that for any $t\in T$ one has $|A\cap (A+t)| \gg |A|^{1-\d}$. 
\label{c:pure_Sidon_intr} 
\end{corollary}
\begin{proof} 
	We apply the first part of  Theorem  \ref{t:Ek} with $\d=\d/2$ and $\eps = \d/16$.
	If the first alternative holds, then we are done in view of Lemma \ref{l:random_Ek} because  $\E^{+}_k (A) \le |A|^{k+\d}$ implies \eqref{f:pure_Sidon_intr}.
	Otherwise there is $H\subseteq \Gr$, $|H| \gtrsim  |A|^{\d(1-\eps)/2}$, $|H+H| \ll |A|^{\eps} |H|$ 
	and there exists $Z\subseteq \Gr$, $|Z| |H| \ll |A|^{1+\eps}$ with  $|(H\dotplus Z) \cap A| \gg |A|^{1-\eps} \,.$ 
	For any $z\in Z$ put $A_z = A \cap (H+z) \subseteq A$. 
	Applying the Cauchy--Schwarz inequality several times,  we have 
\[
	\sum_{x} r_{A-A} (x) r_{H-H} (x) = \E(A,H) \ge \sum_{z\in Z} \E(A_z,H) \gg \sum_{z\in Z} \frac{|A_z|^2 |H|^2}{|A|^\eps |H|}
	\gg
\]
\begin{equation}\label{tmp:08.03_1}
	\gg  
		(|H||Z|)^{-1} |A|^{2-3\eps} |H|^2 \gg |A|^{1-4\eps} |H|^2 \,.
\end{equation}
	Taking $T=\{t \in \Gr ~:~ |A\cap (A+t)| \gg |A|^{1-4\eps} \}$, we obtain from \eqref{tmp:08.03_1} that 
	$$
		|T| \gg |H||A|^{-4\eps} \gg |A|^{\d/8} \,.
	$$  
	This completes the proof. 
	$\hfill\Box$
\end{proof}

\bigskip

Now we are ready to obtain  bound \eqref{f:main_intr1} of Theorem \ref{t:main_intr}. 
We give even two proofs using both statements of our driving Theorem \ref{t:Ek}.

\bigskip 

\begin{proof} 
	Take any $\d<1/2$, e.g., $\d=1/4$ and let $\eps \le \d/4$ be a parameter, which we will choose later. 
	In view of Lemma \ref{l:random_Ek} we see that $\E^{+}_k (A) \le |A|^{k+\d}$ implies 
\begin{equation}\label{tmp:08.03_2} 
	\Sid^{+}_{3k-3} (A) \gg |A|^{\frac{1}{2} + \frac{1-2\d}{2(2k-1)}} = |A|^{\frac{1}{2} + \frac{1}{4(2k-1)}}
\end{equation} 
	and we are done. 
	Here $k=k(\eps)$. 
	Otherwise there is $H\subseteq \F$, $|H| \gtrsim  |A|^{\d(1-\eps)} \ge |A|^{\d/2}$, $|H+H| \ll |A|^{\eps} |H|$ 
	and there exists $Z\subseteq \F$, $|Z| |H| \ll |A|^{1+\eps}$ with  $|(H\dotplus Z) \cap A| \gg |A|^{1-\eps} \,.$ 
	Put $A_* = (H\dotplus Z) \cap A$, $|A_*| \gg  |A|^{1-\eps}$ and we want to estimate $\E^{\times}_{l+1} (A_*)$ or $\hat{\E}_{l+1}^{\times} (A_*)$ for large $l$. 
	After that having a good upper bound for $\E^{\times}_{l+1} (A_*)$ or $\hat{\E}_{l+1}^{\times} (A_*)$,  we apply  Lemma \ref{l:random_Ek} again to find large multiplicative Sidon subset of $A_*$.

	First of all, notice that $|A_* + H| \le |H+H||Z| \ll |A|^\eps |H||Z| \ll |A|^{1+2\eps}$.
	In other words, the set $A_*$ almost does not grow after the addition with $H$.
	Let $Q = A_*+H$, $|Q| \ll   |A|^{1+2\eps}$.  
	Secondly, fix any $\lambda \neq 0$. 
	The number of the solutions to the equation $a_1 a_2 = \lambda$, where $a_1,a_2 \in A_*$ does not exceed 
	$$
		\sigma_\la := |H|^{-2} |\{ h_1,h_2\in H,\, q_1,q_2\in Q ~:~ (h_1-q_1)(h_2-q_2) = \lambda \}| \,.
	$$ 
	The last equation can be interpreted as a question about the number of incidences between points and modular hyperbolas (see \cite{s_Kloosterman}) and for each non--zero $\lambda$ the quantity $\sigma_\lambda$ can be estimated as
	$$
		\sigma_\la \ll |H|^{-2} \cdot |Q| |H|^{2-\kappa} \ll |A|^{1+2\eps} |H|^{-\kappa} \,,
	$$
	see \cite{Brendan_rich} in the case $\F=\R$ and \cite[Theorem 22]{s_Kloosterman} in the case of the prime field.
	Here $\kappa = \kappa(\d)>0$. 
	Recalling that $|H| \gg |A|^{\d/2}$, $|A_*| \gg |A|^{1-\eps}$ and taking any $\eps \le  \d \kappa/100$, we obtain after some calculations that $\sigma_\la  \ll  |A_*|^{1-\d \kappa/4}$.
	Hence taking sufficiently large $l \gg (\d \kappa)^{-1}$, we derive 
\[
	\hat{\E}_{l+1}^{\times} (A_*) = 
	\sum_{\lambda} r^{l+1}_{A_* A_*} (\lambda) \ll  |A_*|^{l+1} + (|A_*|^{1-\d \kappa/2})^l |A_*|^2 \ll |A_*|^{l+1} + |A|^{l+2-\d\kappa l/2} \ll |A_*|^{l+1}\,.
\]
	Applying Lemma \ref{l:random_Ek} and choosing  $\eps \ll l^{-1}$,   we see that 
	$$
		\Sid^\times_{2l} (A) \ge \Sid^\times_{2l} (A_*)
		\gg	
		|A_*|^{\frac{l+1}{2l+1}}
		\gg
		|A|^{\frac{(1-\eps)(l+1)}{2l+1}}
		= 
		|A|^{\frac{1}{2} + \frac{1-2\eps(l+1)}{2(2l+1)}} \gg |A|^{\frac{1}{2}+c} \,, 
	$$
	where $c = c(\d) >0$ is an absolute constant.

	Now let us give the second proof applying the last part of Theorem \ref{t:Ek}. 
	The argument is almost the same but we estimate $\hat{\E}^\times_{l+1} (A')$.
	Again, 	the number of the solutions to the equation $a_1 a_2 = \lambda$, where $a_1,a_2 \in A'$ does not exceed 
	$$
	\sigma_\la := (|P| |A|^{-\eps})^{-2} \cdot |\{ a_1,a_2\in A,\, p_1,p_2\in P ~:~ (a_1-p_1)(a_2-p_2) = \lambda \}| \,.
	$$ 
	By results from \cite{Brendan_rich}, \cite{s_Kloosterman} one has 
	$$
	\sigma_\la \ll (|P| |A|^{-\eps})^{-2} \cdot |A| |P|^{2-\kappa} \ll |A|^{1+2\eps} |P|^{-\kappa} \,,
	$$
	where $\kappa = \kappa(\d)>0$. 
	Again $|P| \gtrsim |A|^\d$ and 
	we can use the arguments as above.
	This 
	concludes 
	the proof. 
$\hfill\Box$
\end{proof}

\bigskip

To complete the proof of Theorem \ref{t:main_intr} we need a simple lemma on upper estimates of sizes of Sidon sets in sumsets.

\begin{lemma}
	Let $A\subseteq \Gr$ be a set, $A=B+C$, and $k\ge 1$ be an integer.  
	Then
	$$
		\Sid^{}_k (A) \le \min \{ |C| \sqrt{k|B|}  + |B|, |B| \sqrt{k|C|}  + |C| \} \,. 
	$$
	More generally, if $r_{B+C} (a) \ge \sigma$ for any $a\in A$, then  
	$$
		\Sid^{}_k (A) \le \sigma^{-1} \min \{ |C| \sqrt{k|B|}  + |B|, |B| \sqrt{k|C|}  + |C| \} \,. 
	$$
\label{l:L_in_sumsets}
\end{lemma}
\begin{proof}
	We give two proofs. 
	The first proof uses theory of graphs and it demonstrates transparently how $\Sid^{}_k$--sets are naturally connected  with Cayley graphs.

	Let $\Lambda$ be the maximal 
	subset of $A$ such that $r_{\Lambda-\Lambda} (x) \le k$ for any $x\neq 0$. 
 	Consider the graph $G= G(V,E)$, where $V$ is the disjoint union of $B$ and $C$, the edge $(b,c) \in E$ iff $b+c\in \Lambda$.
	Moreover, we assume that $|E|=|\L|$ ignoring elements of $\Lambda$, which have several representations as $b+c$. 
	Using the Cauchy--Schwartz inequality, we obtain  
\[
	|\Lambda|^2 \le |B| \sum_{b\in B} \sum_{c,c'\in C} E(b,c) E(b,c') 
	=
\]
\[
	= 
	|B| \sum_{c,c'\in C}\, \sum_{b\in B} E(b,c) E(b,c') = 
	|B| |\L| + |B| \sum_{c,c'\in C,\, c\neq c'}\, \sum_{b\in B} E(b,c) E(b,c') \,.
\] 
	If the last sum over $b$ is at least $k+1$, then we find a 
	complete bipartite subgraph $K_{2,k+1}$
	in $G$ and hence there are different elements $\la_1,\la'_1, \dots, \la_{k+1}, \la'_{k+1} \in \Lambda$ such that $\la_1-\la'_1=\dots = \la_{k+1}-\la'_{k+1}$. The last fact contradicts the assumption that $r_{\Lambda-\Lambda} (x) \le k$ for any $x\neq 0$. 
	Hence 
\[
	|\Lambda|^2 \le |B| \le |B| |\Lambda| + k|B| |C|^2 
\]	
	as required.

	To obtain the second part of our lemma we use a little bit different method. 
	Again, let $\Lambda$ be the set as before. 
	Then by the Cauchy--Schwarz inequality, we get 
\[
	(\sigma |\Lambda|)^2 \le \mathcal{S}^2 := \left( \sum_{x\in \Lambda} r_{B+C} (x) \right)^2 = \left( \sum_{b\in B} r_{\Lambda -C} (b) \right)^2
		\le
		 |B| \sum_{x,y\in \L} |B \cap (x-C) \cap (y-C)| \le
\]
\[
	\le 
		|B| \mathcal{S} + |B| \sum_{x \neq y\in \L} |(x-C) \cap (y-C)| 
		\le
		|B| \mathcal{S} + k |B| |C|^2 \,.
\]
	This completes the proof. 
	$\hfill\Box$
\end{proof}

\bigskip

Now we can easily obtain a non--trivial upper bound for size of maximal Sidon set in any difference set or sumset.

\begin{corollary}
	Let $A\subseteq \Gr$ be a set and $D=A-A$, $S=A+A$. 
	Then for any positive integer $k$ one has $\Sid_k (D) \ll \sqrt{k}  \min \{ |A|^{3/2}, \sqrt{|D|^3} |A|^{-1} \}$ and  
	$$
		\Sid_k (S) \ll \sqrt{k} \min \{ |A|^{3/2}, |A|^{-1} \min\{ |D| \sqrt{|S|} +|S|, |S| \sqrt{|D|} + |D| \}  \} \,.
	$$   
\end{corollary}
\begin{proof} 
	The bound $\Sid_k (D), \Sid_k (S) \ll \sqrt{k} |A|^{3/2}$ follows immediately  from the first part of Lemma \ref{l:L_in_sumsets}. 
	Further it is easy to see (or consult \cite{SS1}) that for any $d\in D$ one has $r_{D-D} (d) \ge |A|$. 
	Applying the second part of Lemma \ref{l:L_in_sumsets} with $A=B=C=D$ and $\sigma = |A|$, we obtain $\Sid_k (D) \ll \sqrt{k |D|^3} |A|^{-1}$.
	To prove the second part of our lemma notice that for any $s\in S$ the following holds $r_{D+S} (s) \ge |A|$.  
	This 
	concludes 
	the proof. 
$\hfill\Box$
\end{proof}

\bigskip

To complete the proof of Theorem \ref{t:main_intr} we need to obtain upper bound \eqref{f:main_intr2}. 
In the case $\F = \R$ 
we put $B=\G$, $C=H\G$, where $\G = \{1,g,\dots, g^{n} \}$, $g\ge 2$ is an integer, 
$\bar{\G} = \{g^{-n}, \dots, g^{-1}, 1,g,\dots, g^{n} \}$, $H= \{ g^{n+1}, g^{2(n+1)}, \dots, g^{n(n+1)} \}$. 
Then $A=B+C = \G + H\G$ is contained  in $\G (1+H \bar{\G})$  and in view of Lemma \ref{l:L_in_sumsets} any additive/multiplicative $k$--Sidon   subset of $A$ has size at most $O_k(|\G|^2)=O_k (|A|^{2/3})$ because as one can easily see $|A|=|\G|^3$. 
Similarly, in the case $\F = \F_p$ we apply Lemma \ref{l:L_in_sumsets} with $B=\G$, $C=H\G$, where $\G\le \F_p^*$, $H\subseteq \F_p^* / \G$ and $|H| = |\G|$ is sufficiently small relatively to $p$.
Then $A:= B+C = \G (1+\G H)$ and hence by  Lemma \ref{l:L_in_sumsets} any additive/multiplicative $k$--Sidon   subset of $A$ has size at most $2\sqrt{k} |\G|^2$. 
To obtain the required bound $|A|\gg |\G+H\G|$ for an appropriate $H$ one can use the random choice (we leave the details to the interested reader).

\section{On $B^\circ_2 [g]$--sets}

In the previous section we 
have obtained some results about the family of sets $S$ with 
\begin{equation}\label{def:B_c}
	r_{S-S} (x) \le g\,,  \quad \quad  \forall x\in \Gr \,,  \quad x\neq 0 \,.   
\end{equation}
This class of sets were introduced by Erd\H{o}s in \cite[page 57]{E_Extr} 
(also, see \cite{EH}) 
and we denote this family as $B^\circ_2 [g]$ (Erd\H{o}s used the symbol $B^{'(k)}_2$). 
It was said in \cite{E_Extr} that "V.T. S\'os and I considered $B^{'(k)}_2$ sequences... We could not decide whether there is a $B^{'(k)}_2$ sequence which is not the union of a finite number of Sidon sequences." According to the author's knowledge this paper of Erd\H{o}s and S\'os was not published. 
The question 
from \cite{E_Extr} is 
a nice problem  of Erd\H{o}s, which is open and if it has 
a negative 
answer, then the original question of 
Klurman--Pohoata 
would be closed thanks to our Theorem \ref{t:main_intr}.  
Let us underline it one more time that it is possible to construct sets $S$ with bounded $r_{S+S} (x)$ which are not the union of a finite number of Sidon sequences (see \cite{E_Extr} and \cite{AE}).
It seems like 
condition \eqref{def:B_c} has another nature 
and that   is why 
we devote this section studying some further properties of $B^\circ_2 [g]$--sets. 
To see that this is a special family, notice that, for example, the condition $r_{S-S} (x) \ll 1$, $x\neq 0$ has an interpretation in terms of the Cayley graph of $S$ but $r_{S+S} (x) \ll 1$ 
cannot be expressed in terms of any Cayley graph.

First of all, notice that if $S$ is a random subset of $[N]$, which was obtained piking elements from $[N]$ independently at random with probability $q \sim N^{-1/2}$, then 
with probability $1-o(1)$ one has $r_{S-S} (x) \ll \log N$ see, e.g., \cite[Lemma 4.3]{DSSS} and a similar lower bound for the function $r_{S-S} (x)$ takes place.
Thus for any fixed $g$ sets belonging to $B^\circ_2 [g]$ 
are far 
from random sets.

Secondly, as it was noted in the proof of Lemma \ref{l:L_in_sumsets} a set $S$ belongs to the family $B^\circ_2 [g]$ iff its Cayley graph $\Cay(S,\Gr)$ has no complete bipartite subgraphs $K_{2,g+1}$.  
Recall that the vertex set of $\Cay(A,\Gr)$ is $\Gr$ and $(x,y)$ is an edge of $\Cay(A,\Gr)$ iff $x-y\in A$. 
Another equivalent interpretation of \eqref{def:B_c} is 
\begin{equation}\label{def:B_c_2}
	|S \cap (S + x_1) \cap \dots \cap (S+x_g)| \le 1 \quad \quad  \mbox{for all distincts and non--zero } \quad x_1,\dots, x_g \in \Gr
	\,. 
\end{equation}
This reinterpretation of $B^\circ_2 [g]$--sets says that 
the considered 
family is naturally connected with the higher energies \cite{SS1} (also, see the second formula in definition \eqref{def:E_k}). 
Also, if we put for an arbitrary  $A\subseteq \Gr$ 
$$\D_g (A) := \{ (a,\dots,a) \in A^g ~:~ a\in A\} \subseteq \Gr^g \,, $$
then \eqref{def:B_c_2} means that the sum $S^g$ and $\D_g (S)$ is direct, in other words, $S^g$ and $\D_g (S)$ form a {\it co--Sidon pair}, see \cite{DSSS}. 
Further formula \eqref{def:B_c_2} suggests us the following definition for $B^\circ_k [g]$ sets
\begin{equation}\label{def:B_c_3}
	|S \cap (S + x_1) \cap \dots \cap (S+x_g)| < k \quad \quad  \mbox{for all distincts and non--zero } \quad x_1,\dots, x_g \in \Gr \,.
\end{equation}
Similarly, $S\in B^\circ_k [g]$ iff  $\Cay(S,\Gr)$ does not contain  $K_{k,g+1}$ or, in other words, $S$ does not contain any sumsets $X+Y$, where $|X|=g+1$, $Y=k$.   
Since the number of edges in $\Cay(S,\Gr)$ equals $|S||\Gr|$ for any finite group $\Gr$, it follows that to estimate size of $S$ 
it is enough to bound the number of edges in $K_{k,g+1}$--free graphs. 
Such results are discussed in detail  in survey \cite{FS}, for example.   
Of course there is a direct approach to estimate cardinality of $S \in B^\circ_2 [g]$. 
Namely, if $S \in B^\circ_2 [g]$ belongs to a  group $\Gr$ of size $N$, then, obviously, 
\[
	|S|^2 = \sum_{x} r_{S-S} (x) \le |S| + g (N-1) 
\]
and hence 
\begin{equation}\label{f:Bg_G}
	|S| < \sqrt{gN} + 1 \,.
\end{equation}
Similarly, if $S\in B^\circ_k [g]$ and $S$  belongs to a group $\Gr$ of size $N$, then 
\[
	|S|^{g+1} = \sum_{x_1,\dots,x_g} |S \cap (S + x_1) \cap \dots \cap (S+x_g)| < kN^g + \binom{g+1}{2} |S|^g
\]
and thus $|S|<k^{\frac{1}{g+1}} N^{\frac{g}{g+1}} + \binom{g+1}{2}$. 
Of course in the case when $S\subseteq [N]$ there are no such good upper bounds for size of $S$ even if $S$ is a classical Sidon set.   
Nevertheless, we easily obtain a generalization of Linstrom's result \cite{Linstrom} for $B^\circ_2 [g]$--sets in the segment (as well as for $B^\circ_k [g]$--sets but it is not the main topic of our paper, for better bounds see \cite{Ch[g]}).

\begin{proposition}
	Let $S \subseteq [N]$ belongs to the family $B^\circ_2 [g]$. 
	Then 
\begin{equation}\label{f:Bg_size}
	|S| < \sqrt{gN} + (gN)^{1/4} + 1 \,.
\end{equation}
	Generally, if $S\in B^\circ_k [g]$, then 
\begin{equation}\label{f:Bg_size2}
	|S| <  k^{\frac{1}{g+1}} N^{\frac{g}{g+1}}+ \binom{g+1}{2}^{\frac{1}{g+1}} k^{\frac{g}{(g+1)^2}} N^{\frac{g^2}{(g+1)^2}} +  1\,.
\end{equation}
\label{p:Bg_size}
\end{proposition}
\begin{proof} 
	We use the method from \cite{Green}. 
	Let $u=[N^{3/4} g^{-1/4}]$ be a parameter and $I = [u]$. 
	Embed  $S$ into $\Z/(N+u)\Z$. It is easy to check that for any $x\in [-u,u] \setminus \{0\}$ one has $r_{S-S} (x) \le g$, where now $x$ runs over  $\Z/(N+u)\Z$. 
	Hence using the Cauchy--Schwarz inequality to estimate the common energy $\E(S,I)$, we get
\[
	\frac{|S|^2 u^2}{N+u}\le \E(S,I) = \sum_{x} r_{S-S} (x) r_{I-I} (x) < |S| u + g u^2   
\]
	or, in other words, 
\[
	g u^2 + u (gN+|S|-|S|^2) + |S| N > 0 
\]
	and substituting $u=[N^{3/4} g^{-1/4}] < N^{3/4} g^{-1/4}$ and $|S| = \sqrt{gN} + (gN)^{1/4} + C$, we find after some calculations that the condition  $C\le 1$ is enough. 
	To obtain \eqref{f:Bg_size2} we use a similar argument (with another parameter $u$, of course) to estimate an analogue of the higher common energy of $S$ and $I$, namely,  
\begin{equation}\label{tmp:10.03_1}
	\frac{(|S|u )^{g+1}}{(N+u)^{g}} \le \sum_{x} r_{S-I}^{g+1} (x) = \sum_{x_1,\dots,x_g} |S\cap (S+x_1) \cap \dots \cap (S+x_g)| 
	|I\cap (I+x_1) \cap \dots \cap (I+x_g)|
	<
\end{equation}
\[ 
	<
	k u^{g+1} + \binom{g+1}{2} \sum_{x} r_{S-I}^{g} (x)
	\le
	k u^{g+1} + \binom{g+1}{2} |S|^g u := k u^{g+1} + C_g |S|^g u \,.
\]
	We have used in \eqref{tmp:10.03_1} that all variables $x_j$ belong to $[-u,u]$. 
	Finally, in the case of $B^\circ_k [g]$--sets 
	an appropriate 
	choice of the parameter $u$ is $u=[C^{1/(g+1)}_g k^{-1/(g+1)^2} N^{1-g/(g+1)^2}]$ and after some calculations we obtain \eqref{f:Bg_size2} (we roughly estimate the third term in this formula by one and actually, we do not optimize the constant in the middle term of \eqref{f:Bg_size2}, taking it for the simplest way to check).
	This completes the proof. 
$\hfill\Box$
\end{proof}


\bigskip 

As for lower bounds on size of maximal subsets of $B^\circ_k [g]$,  again one can  consult \cite{FS} to find correspondent lower bound for the number of edges in graphs having no $K_{s,t}$. 
Our graphs must be {\it Cayley graphs} and such constructions are known for $K_{2,2}$, $K_{3,3}$ (Brown's construction, see \cite{Brown}) and for $K_{s,t}$, $t\ge s!+1$ (so--called, norm--graphs), see   \cite{FS}. 
As for $K_{2,t}$, $t>2$ one can obtain a result similar to \cite[Theorem 1.6]{CRV}. 
Namely, define for any $g$ the quantity 
\[
	\a_g = \limsup_{N\to \infty} \frac{ \max\{ |S| ~:~ S\subseteq \Z/N\Z\,, S\in B^\circ_2 [g] \} }{\sqrt{N}} \,.
\]

\begin{theorem}
	We have
\[
	\a_g = \sqrt{g} + O(g^{3/10}) \,.
\]
\end{theorem}
\begin{proof} 
	Actually, our argument almost coincides with the approach  of the proof \cite[Theorem 1.6]{CRV}, so we omit the details. 
	
	By the method of \cite[section 4]{CRV} it is enough to construct a set $A\subseteq (\Z/p\Z)^2$ (here $p$ is a  prime number) such that $A \in  B^\circ_2 [g]$ with $g=k^2 + O(k^{3/2})$ and $|A| = kp +O(k)$.  
	We put $A=\bigsqcup_{u\in U} A_u$, where $U$ is an appropriate arithmetic progression, $U=t+[k]$ and for any non--zero $u \in \Z/p\Z$ we define 
	$$A_u = \{ (x, x^2/u) ~:~ x\in \Z/p\Z \} \subset  (\Z/p\Z)^2 \,.$$ 
	Clearly, $|A|=kp-k+1$. 
	Let $r_{u,v} (x) = r_{A_u-A_v} (x)$.
	Our task is to estimate for any $x$ the sum 
	$$r_{A-A} (x) \le \sum_{u,v\in U} r_{u,v} (x)  \,.$$
	By \cite[Lemma 3.2]{CRV} one has $r_{u,v} (x) + r_{u',v'} (x) =2$, provided $u+v=u'+v'$ and $\left(\frac{uvu'v'}{p}\right) = -1$.   
	Using this lemma and acting exactly as on pages 2794--2795 of \cite{CRV}, we find 
\[
	r_{A-A} (x) \le k^2 + \sum_{|l| < k} \left| \sum_{i+j=k+1+l} \left( \frac{(t+i)(t+j)}{p} \right)\right| \,.
\]
	After that taking the summation over $t$ (to find an appropriate $U$) and applying the Cauchy--Schwarz inequality and Weil's bound on sum of Legendre symbols, we obtain   the required estimate (see the rest of the argument from  \cite{CRV}). 
	This completes the proof. 
$\hfill\Box$
\end{proof}

\bigskip 

Finally, let us mention that the construction of Linstr\"om \cite{Linstrom_g-h} works for $B^\circ_2 [g]$, as well as for usual Sidon sets. 
Namely, having a Sidon set $A$ in $[(N-g)/g]$ one can see that the set $S:= g\cdot A + \{ 0,1,\dots, g-1 \} \subseteq [N]$ belongs to the family $B^\circ_2 [g]$ and hence it gives us another construction of sets $S\in B^\circ_2 [g]$, $S\subseteq [N]$ of size $\sqrt{gN} (1+o(1))$. 

\bigskip 

We continue  this section
considering heritability properties of $B^\circ_2 [g]$--sets.
For different 
$x_1\dots,x_s$ and any set $A\subseteq \Gr$  put $A_{X} := (A+x_1) \cap \dots \cap (A+x_s)$, where $X= \{ x_1,\dots,x_s \}$.  
Taking the same set $X_1=\dots = X_l = X$ in inequality \eqref{f:sub_Sidon} below, one can see that any $B^\circ_k [g]$--set generates $B^\circ_k [g_1]$--sets with a smaller $g_1$. 
In particular, in Brown's construction \cite{Brown} of the set $S\in B^\circ_2 [2]$, this set $S$ is a (almost disjoint) union of classical Sidon sets $S_w = S \cap (S+w)$, $w\in (S-S)\setminus \{0\}$.

\begin{proposition}
	Let $S \in B^\circ_k [g]$, $l\ge 2$ and take any sets $X_1,\dots,X_l$ with  $|X_1|+\dots + |X_l| \ge g+\binom{l}{2}+1$ and such that any $(X_i,X_j)$ forms a co--Sidon pair.  
	Then 
\begin{equation}\label{f:sub_Sidon}
	|S_{X_1} \cap (S_{X_2} + z_1) \cap \dots \cap (S_{X_l} + z_{l-1})| < k  \quad \quad \mbox{ for any different } \quad \quad z_1,\dots, z_{l-1} \neq 0 \,. 
\end{equation}
	In particular, if $S \in B^\circ_2 [2]$ and $\Gr$ has no elements of order two, then for any non--zero $w$ the set $S_w = S \cap (S+w)$ is a Sidon set.  
\label{p:sub_Sidon}
\end{proposition}
\begin{proof} 
	Let $X_j = \{ x^{(j)}_1, \dots, x^{(j)}_{s_j}\}$, $j\in [l]$.  
	For an arbitrary  $z \in \Gr$, we have 
$$
	\mathcal{S} := S_{X_1} \cap (S_{X_2} + z_1) \cap \dots \cap  (S_{X_l} + z_{l-1}) = 
$$
$$	
	(S+x^{(1)}_1) \cap \dots \cap (S+x^{(1)}_{s_1}) \cap 
	(S+x^{(2)}_1+z_1) \cap \dots \cap (S+x^{(2)}_1+z_1) \cap \dots \cap 
	(S+x^{(l)}_1+z_{l-1}) \cap \dots \cap (S+x^{(l)}_1+z_{l-1})
	\,.
$$
	Since $\sum_{j=1}^l s_j \ge g+\binom{l}{2}+1$, then either $|\mathcal{S}| < k$ by \eqref{def:B_c_2}, \eqref{def:B_c_3} or we have for some indices $x^{(i')}_i + z_i = x^{(j')}_j + z_j$ and  $x^{(i'')}_i + z_i = x^{(j'')}_j + z_j$.
	The last alternative implies that $x^{(i')}_i - x^{(j')}_j  = x^{(i'')}_i - x^{(j'')}_j$ but $X_i$ and $X_j$ form a co--Sidon pair by the assumption.

	In the case $S_w = S \cap (S+w)$ our ground set $X$ is $\{0,w\}$ and it is easy to see that this is a Sidon set. 
 	This completes the proof. 
$\hfill\Box$
\end{proof}

\bigskip 

In Proposition \ref{p:sub_Sidon} we assume that each  $(X_i,X_j)$ forms a co--Sidon pair. 
Again (and this is in the spirit of this section) one can make more general assumptions to the intersections of some shifts of the sets $X_j$ as in \eqref{def:B_c_2}, \eqref{def:B_c_3}  to obtain higher order Sidon sets.

\bigskip

We finish this section discussing the 
tightness 
of Lemma \ref{l:random_Ek}. 
In the next proposition we show that any set $A$ contains rather large (in terms of its energy) subset with controllable size of the maximal $B^\circ_k [g]$ subset.

\begin{proposition}
	Let $A\subseteq \Gr$ be a set. 
	Then there is $A_* \subseteq A$ such that $\E_{g+1} (A_*) \gg_g \E_{g+1} (A)$ and  
	the maximal $B^\circ_k [g]$ subset of $A_*$ has size $O_g (k^{1/(g+1)} |A|^2  \E^{-1/(g+1)}_{g+1} (A))$.\\ 
	In particular, for any $n$ one has $\Sid_n (A_*) \ll n^{1/2} |A|^2 \E^{-1/2} (A)$. 
\label{p:A_*_Sidon} 
\end{proposition}
\begin{proof} 
	We have 
	\[
	\E_{g+1} (A) = \sum_y  r^{g+1}_{A-A} (y) =  \sum_{a\in A} \sum_{x\in A} r^g_{A-A} (x-a) \,.  
	\]
	Put $A_* = \{ a\in A ~:~  \sum_{x\in A} r^g_{A-A} (x-a) \ge \E_{g+1} (A) /(2|A|) \}$.
	Then
	\[
	\E_{g+1} (A) \le 2 \sum_{a\in A_*} \sum_{x\in A} r^g_{A-A} (x-a) = 2 \sum_y r^{g}_{A-A} (y) r^{}_{A-A_*} (y) 
	\] 
	and using the H\"older inequality several times, we obtain
	\begin{equation}\label{tmp:18.03_1}
	\E_{g+1} (A) \le 2^{g+1} \sum_y r^{g+1}_{A-A_*} (y) \le 2^{g+1} (\E_{g+1} (A_*) \E^{2g+1}_{g+1} (A))^{1/(2g+2)} 
	\end{equation}
	or, in other words, $\E_{g+1} (A_*) \gg_g \E_{g+1} (A)$. 
	Now let $\Lambda$ be the  maximal $B^\circ_k [g]$ subset of $A_*$. 
	Then by the definition of the set $A_*$, one has 
	\begin{equation}\label{f:L_A_*}
	2^{-1} |\L| \E_{g+1} (A) |A|^{-1} \le \sum_y r^{g}_{A-A} (y) r^{}_{A-\L} (y) \,.
	\end{equation}
	Hence as in \eqref{tmp:18.03_1}
	\[
	2^{-(g+1)} |\L|^{g+1} \E_{g+1} (A) |A|^{-(g+1)} \le  \sum_y r^{g+1}_{A-\L} (y) 
	=
	\]
	\[
	= 
	\sum_{x_1,\dots,x_g} |A\cap (A+x_1) \cap \dots \cap (A+x_g)|  |\L\cap (\L+x_1) \cap \dots \cap (\L+x_g)| \,.
	\]
	Applying the fact that $\L \in B^\circ_k [g]$, we get 
	\[
	2^{-(g+1)} |\L|^{g+1} \E_{g+1} (A) |A|^{-(g+1)} \le k|A|^{g+1} + \binom{g+1}{2} |A| |\L|^g \,.
	\]
	Noting that the second term in the last formula is negligible, we obtain the required result. 
	$\hfill\Box$
\end{proof}

\begin{remark}
	From \eqref{f:L_A_*} it follows that the conclusion of Proposition \ref{p:A_*_Sidon} remains true for a wider family of sets $\L$, namely, for $\L$ with  $\E_{g+1}(\L) \ll |\L|^{g+1}$.
	This class of sets and its connection with Sidon sets was discussed in \cite[Section 3.2]{R-N_W}. 
\end{remark}

\bigskip

\noindent{I.D.~Shkredov\\
	Steklov Mathematical Institute,\\
	ul. Gubkina, 8, Moscow, Russia, 119991}
\\
and
\\
IITP RAS,  \\
Bolshoy Karetny per. 19, Moscow, Russia, 127994\\
and 
\\
MIPT, \\ 
Institutskii per. 9, Dolgoprudnii, Russia, 141701\\
{\tt ilya.shkredov@gmail.com}

\end{document}